\documentclass{amsart}

\usepackage{amssymb}
\usepackage{graphicx}
\usepackage{MnSymbol}

\usepackage{amsthm}

\title{On subgroups of first homology}
\author{Samuel M. Corson}

\theoremstyle{definition}\newtheorem{theorem}{Theorem}
\theoremstyle{definition}\newtheorem*{A}{Theorem \ref{directsum}}
\theoremstyle{definition}\newtheorem*{B}{Theorem \ref{Borel}}

\theoremstyle{definition}\newtheorem{bigtheorem}{Theorem}

\theoremstyle{definition}
\theoremstyle{definition}
\theoremstyle{definition}\newtheorem{definition}[theorem]{Definition}
\theoremstyle{definition}
\theoremstyle{definition}\newtheorem{example}{Example}
\theoremstyle{definition}
\theoremstyle{definition}
\theoremstyle{definition}\newtheorem{lemma}[theorem]{Lemma}
\theoremstyle{definition}
\theoremstyle{definition}
\theoremstyle{definition}
\theoremstyle{definition}
\newtheorem*{question*}{Question}
\newtheorem*{theorem*}{Theorem}
\newtheorem*{corollary*}{Corollary}

\def\pmc#1{\setbox0=\hbox{#1}
    \kern-.1em\copy0\kern-\wd0
    \kern.1em\copy0\kern-\wd0}

\newcommand{\diam}{\operatorname{diam}}

\newcommand{\cl}{\operatorname{cl}}

\newcommand{\tor}{\operatorname{Tor}}
\newcommand{\torfree}{\operatorname{Torfree}}
\newcommand{\Div}{\operatorname{Div}}
\newcommand{\red}{\operatorname{Red}}
\newcommand{\inffree}{\operatorname{InfFree}}
\newcommand{\Hom}{\operatorname{Hom}}
\newcommand{\rank}{\operatorname{rank}}
\newcommand{\h}{\overline{H}_1}
\newcommand{\card}{\operatorname{card}}
\newcommand{\Po}{\mathcal{P}}

\begin{document}

\address{Mathematics Department\\
1326 Stevenson Center\\
Vanderbilt University\\
Nashville, TN 37240\\
USA}			

\email{samuel.m.corson@vanderbilt.edu}

\maketitle

\begin{abstract}  We prove several new theorems regarding first homology.  Some dichotomies for first homology of Peano continua are presented, as well as a notion of strong abelianization for arbitrary path connected metric spaces.  We also show that the fundamental group of the Hawaiian earring has Borel subgroups of almost all multiplicative and additive types.
\end{abstract}

\begin{section}{Introduction}  Much of the recently developed theory for fundamental groups of path connected Polish spaces can be used in studying the first singular homology of  such spaces.  One or two results in this vein have already been explored (see \cite{CoCo}).  The purpose of this note is to present deeper explorations in the study of first homology.  We describe the organization of this note.

In Section \ref{Preliminaries} we will present some of the preliminary definitions and tools that will be used in our study.  The notion of analytic subgroup of the fundamental group will be reviewed, as well as statements of theorems that will be used.  In Section  \ref{Structuretheorems} we will present the main dichotomy theorems concerning first homology of Peano continua (locally path connected, path connected metrizable compact spaces).  Among these theorems  is the following (with $\torfree(A)$ denoting $A/\tor(A)$):

\begin{bigtheorem}\label{directsum}  Let $X$ be a Peano continuum.  If $\card(\torfree(H_1(X)))<2^{\aleph_0}$ then $H_1(X)$ is a direct sum of cyclic groups and $\tor(H_1(X))$ is of bounded exponent.

\end{bigtheorem}

In Section \ref{Astrongabelianization} we introduce a very natural definition of strong abelianization of the fundamental group.  Some desirable properties of this definition are established (functoriality over countable products and shrinking wedges).

In Section \ref{Borelsubgroups} we show that the Hawaiian earring $E$ has normal subgroups of all but a couple of the multiplicative and additive Borel pointclasses, as well as a normal subgroup that is analytic and not Borel.  In particular we show:

\begin{bigtheorem}\label{Borel}  The fundamental group $\pi_1(E)$ has normal subgroups of the following types:

\begin{enumerate}\item true $\Sigma_{\gamma}^0$ for $\omega_1>\gamma \geq 2$

\item true $\Pi_{\gamma}^0$  for each $\omega_1 > \gamma \neq 2$

\item true analytic
\end{enumerate}

\end{bigtheorem}

\noindent This is done by demonstrating a correspondence with the fundamental group of the infinitary torus, whose fundamental group is abelian and therefore isomorphic to its first homology.

\end{section}

\begin{section}{Preliminaries}\label{Preliminaries}

We introduce some of the terminology and notation for this paper, assuming some familiarity with the first fundamental group, $\pi_1$.  Let $(X, d)$ be a path connected metric space and $x\in X$.  Let $L_x$ denote the set of all loops based at $x$ and metrize $L_x$ by letting the distance between loops $l_0$ and $l_1$ be given by $\sup_{s\in [0,1]}d(l_0(s), l_1(s))$.  Where there are several spaces in the context we add a second subscript, $L_{x, X}$, to emphasize the space in which $x$ lies.  Let $*$ denote the concatenation operation for loops, so that by definition $[l_0][l_1] = [l_0*l_1]$ for $[l_0], [l_1]\in \pi_1(X, x)$.  If $X$ is separable (respectively a complete metric space) then $L_x$ is separable (completely metrized by the $\sup$ distance function).

We say a subgroup $G\leq \pi_1(X, x)$ is $\Po$ provided $\bigcup G$ is a $\Po$ subset of $L_x$, where $\Po$ is some topologically defined predicate.  In other words $G$ is open, closed, Borel, etc. provided the set of loops of $G$, $\bigcup G$, is respectively open, closed, Borel, etc.  in the loop space.  Recall that a \textbf{Polish space} is a separable, completely metrizable space.  As was noted, if $X$ is Polish then so is $L_x$.  If $Z$ is Polish we say that $Y\subset Z$ is \textbf{analytic} if there exists a Polish space $W$ and continuous function $f:W \rightarrow Z$ such that $f(W) = Y$.  All Borel subsets of a Polish space are also analytic, and the class of analytic sets is closed under countable unions, countable intersections, continuous images in a Polish space and continuous preimages.  The following lemma combines the statements of Lemmas 3.3-5 and Theorems 3.12, 3.13 and C from \cite{Co}:

\begin{lemma}\label{biglemma}  The following hold:
\begin{enumerate}

\item  If $X$ is path connected Polish and $K \subset L_x$ is analytic then the set $[K] \subset L_x$ of loops homotopic to an element of $K$ is analytic, and the generated subgroup $\langle [K]\rangle$ and the normal generated subgroup $\langle \langle [K] \rangle\rangle$ are both analytic as subgroups of $\pi_1(X, x)$.

\item  If $X$ is a path connected, locally path connected Polish space and $G \unlhd \pi_1(X, x)$ is analytic then either $\card(\pi_1(X, x)/G) \leq \aleph_0$ (in case $G$ is open) or  $\card(\pi_1(X, x)/G) = 2^{\aleph_0}$ (in case $G$ is not open).

\item  If $X$ is a Peano continuum and $G \unlhd \pi_1(X, x)$ is analytic then either $\pi_1(X, x)/G$ is finitely generated (in case $G$ is open) or of cardinality $2^{\aleph_0}$ (in case $G$ is not open).

\item  If $X$ is a Peano continuum there does not exist a strictly increasing sequence $\{G_n\}_{n\in\omega}$ of normal analytic subgroups of the fundamental group such that $\pi_1(X, x) = \bigcup_{n\in\omega} G_n$.

\end{enumerate}

\end{lemma}

The following is a consequence of  Lemma 3.14 in \cite {Co}:

\begin{lemma}\label{smallloopslemma}  Suppose $X$ is a Peano continuum and $N\unlhd \pi_1(X, x)$ is such that $\bigcup N = \bigcup_{n = 0}^{\infty} N_n$ with each set $N_n$ closed under inverses and homotopy and containing the trivial loop, and $N_n*N_{m}\subseteq N_{n+m}$.  If each $N_n$ is analytic and for each $n\in \omega$ there exist loops at $x$ of arbitrarily small diameter not contained in $N_n$, then $\pi_1(X, x)/N$ is of cardinality $2^{\aleph_0}$.
\end{lemma}

Recall the Hurewicz homomorphism $h:\pi_1(X, x) \rightarrow H_1(X)$ which takes a loop to its homology class.  Provided $X$ is path connected (which will be our assumption throughout this note) the map $h$ is onto and the kernel of $h$ is the commutator subgroup $[\pi_1(X, x), \pi_1(X, x)]$, and so $H_1(X)$ is isomorphic to the abelianization of the fundamental group.  It is known that if $X$ is path connected Polish then $[\pi_1(X, x), \pi_1(X, x)]$ is an analytic subgroup of $\pi_1(X, x)$ (see \cite{CoCo} or \cite{Co}).

\end{section}

\begin{section}{Structure theorems}\label{Structuretheorems}
We present some structure theorems regarding $H_1(X)$ where $X$ is a Peano continuum.  We mention that a couple of these results have slight generalizations to path connected locally, path connected Polish spaces (obtained by applying Lemma \ref{biglemma} part (2) rather than part (3)).  In particular, one has modifications of Theorems \ref{reduced} and \ref{inffree} which assume $X$ is a path connected, locally path connected Polish space and conclude that the quotient mentioned in the theorem is either of cardinality $\leq \aleph_0$ or $2^{\aleph_0}$.

The following theorem appeared in \cite{CoCo}:
\begin{theorem*}  If $X$ is a path connected, locally path connected Polish space then $H_1(X)$ is either of cardinality $\leq \aleph_0$ or $2^{\aleph_0}$.  In case such an $X$ is also compact, then $H_1(X)$ is either a finite direct sum of cyclic groups or of cardinality $2^{\aleph_0}$.
\end{theorem*}

The upshot of the proof of this result is that $H_1(X)$ is uncountable precisely when $X$ has arbitrarily small loops which are not nulhomologous.  We give a theorem which can be interpreted as a type of small loop compactness, first giving a definition:

\begin{definition}  Recall that given a group $G$ and $g\in G$ the \textbf{commutator length} of $g$, which we will denote $\cl(g)$, is the smallest number $n$ such that $g$ can be written as a product of $n$ commutators.  We write $\cl(g) = \infty$ in case $g$ is not in the commutator subgroup of $G$.
\end{definition}

\begin{theorem}\label{smallloopcompactness}  If $X$ is a Peano continuum with $H_1(X)$ of cardinality $< 2^{\aleph_0}$ then there exists $\epsilon>0$ and $N\in \omega$ such that any loop of of diameter less than $\epsilon$ is of commutator length $\leq N$.
\end{theorem}

In other words, if there exist arbitrarily small loops of arbitrarily long commutator length, then $\card(H_1(X)) = 2^{\aleph_0}$.  Thus arbitrarily small loops of infinite commutator length exist whenever there are arbitrarily small loops of arbitrarily large commutator length.

\begin{proof}  Suppose the conclusion fails.  We may then pick a sequence of loops $\{l_n\}_{n\in \omega}$ such that $\diam(l_n)\leq 2^{-n}$ and the commutator length of $[l_n] \in \pi_1(X, x_n)$ is greater than $n$.  As $X$ is compact and locally path connected, we may pass to a subsequence if necessary and assume that all loops are based at the same point, say $x$.  Letting $N_n$ be the set of all loops at $x$ which are of commutator length at most $n$, we have that $\bigcup [\pi_1(X, x), \pi_1(X, x)] = \bigcup_{n\in \omega} N_n$.  It is clear that $N_n*N_{m} \subseteq N_{m+n}$ and so we shall be done by Lemma \ref{smallloopslemma} if we show that the set $N_n$ of products of $n$ or fewer commutators is analytic.

To see that $N_n$ is analytic we let $f: \prod_{2n }L_x \rightarrow L_x$ be the map $$f(l_0, l_1, \ldots , l_{2n - 1}) = l_0*l_1*(l_0)^{-1}*(l_1)^{-1}*\cdots *l_{2n-2}*l_{2n-1}*(l_{2n-2})^{-1}*(l_{2n-1})^{-1}$$ and notice that $[f (\prod_{2n }L_x)] = N_n$ and we are done by Lemma \ref{biglemma} part (1).
\end{proof}

We move on to some other structure theorems for first homology of a Peano continuum.  If $A$ is an abelian group let $\tor(A)$ denote the subgroup of $A$ consisting of the torsion elements.  Let $\torfree(A)$ denote the quotient $A/\tor(A)$.  It is not always the case that $\tor(A)$ is a direct summand, even if $A = H_1(X)$ for some Peano continuum $X$, by the following example (adapted from an example in \cite{Fu}):

\begin{example}\label{primeexample}  Let $P$ be the set of primes and for each $p\in P$ let $X_p$ be a Peano continuum with fundamental group isomorphic to $\mathbb{Z}/p$.  Then $X = \prod_{p\in P} X_p$ has fundamental group isomorphic to $\prod_{p\in P} \mathbb{Z}/p$, so $H_1(X) \simeq \prod_{p\in P} \mathbb{Z}/p$ as well.

Suppose $a\in \prod_{p\in P} \mathbb{Z}/p$ is torsion.  If the $p$ coordinate of $a$ is nonzero then $p$ divides the order of $a$, so in particular $a$ must have finite support.  Conversely any finite supported element is torsion, so $\tor(\prod_{p\in P} \mathbb{Z}/p) = \bigoplus_{p\in P} \mathbb{Z}/p$.  Now for $a = (1,1,1,,\ldots) \in \prod_{p\in P} \mathbb{Z}/p$ we have that for any $n\in \omega$ there exists $b \in \prod_{p\in P} \mathbb{Z}/p$ such that $b^n = a$ in $\prod_{p\in P} \mathbb{Z}/p/\bigoplus_{p\in P} \mathbb{Z}/p$.  This is seen by letting $b(p)$ be such that $b(p)^n = 1\mod p$ for all $p$ that do not divide $n$ and $b(p) = 0$ otherwise.  In $\prod_{p\in P} \mathbb{Z}/p$ there is no nontrivial element which has all roots (for example a $p$-th root would not exist if the $p$-th coordinate is not zero).  Thus $\bigoplus_{p\in P} \mathbb{Z}/p$ cannot be a direct summand.
\end{example}

\begin{theorem}\label{modtor}  If $X$ is a Peano continuum then $\torfree(H_1(X))$ is a finite rank free abelian group or of cardinality $2^{\aleph_0}$.
\end{theorem}

\begin{proof}  It suffices to show that the kernel of the map $$\phi: \pi_1(X) \rightarrow H_1(X)/\tor(H_1(X))$$ is analytic.  Notice that $l\in \bigcup \ker(\phi)$ if and only if $(\exists n \in \omega\setminus \{0\})[l^n \in \bigcup [\pi_1(X), \pi_1(X)]]$.  Since $\bigcup [\pi_1(X), \pi_1(X)]$ is analytic and each map $l \mapsto l^n$ is continuous we have that $\ker(\phi)$ is a countable union of continuous preimages of analytic sets, and is therefore analytic.  If $\torfree(H_1(X))$ is not of cardinality $2^{\aleph_0}$ we know it is a finitely generated torsion-free group, and therefore free abelian of finite rank.
\end{proof}

In case $\torfree(H_1(X))$ is a finite rank free abelian group we get the splitting of the short exact sequence of abelian groups: $$0 \rightarrow \tor(H_1(X)) \rightarrow H_1(X) \rightarrow \torfree(H_1(X)) \rightarrow 0$$ so that $H_1(X) \simeq \tor(H_1(X)) \oplus \bigoplus_{m=0}^n \mathbb{Z}$.

Recall that a torsion group is a $p$-group if the order of every element is divisible by $p$.  A basic fact about torsion abelian groups is the following (Theorem 2.1 in \cite{Fu}):

\begin{theorem*}  Every abelian torsion group $A$ may be decomposed into a direct sum of the $p$-groups $A_p \leq A$ where $A_p$ is the subgroup of elements whose order is divisible by $p$.
\end{theorem*}

\begin{A}  Let $X$ be a Peano continuum.  If $\card(\torfree(H_1(X)))<2^{\aleph_0}$ then $H_1(X)$ is a direct sum of cyclic groups and $\tor(H_1(X))$ is of bounded exponent.
\end{A}

\begin{proof}  Let $P = \{p_0, p_1, \ldots\}$ be the set of prime numbers.  Notice that for each prime $p$ the $p$-subgroup $H_1(X)_p \leq H_1(X)$ is such that $[l]\in \pi_1(X)$ maps to $H_1(X)_p$ if and only if $(\exists k\in \omega)[l^{p^k} \in \bigcup [\pi_1(X), \pi_1(X)]]$.  Then the kernels of each of the maps $\phi_p: \pi_1(X) \rightarrow H_1(X)/ H_1(X)_p$ are analytic subgroups.  Writing $H_1(X) \simeq \tor(H_1(X)) \oplus \bigoplus_{m=0}^k\mathbb{Z}$ we select loops $l_0, \ldots, l_k$ such that $[l_k]$ generates the $k$-th copy of $\mathbb{Z}$ in $\bigoplus_{m=0}^k\mathbb{Z}$.  Then the map $\pi_1(X) \rightarrow H_1(X)/\bigoplus_{m=0}^k\mathbb{Z}$ has kernel which is precisely $\langle \{[l_0], \ldots, [l_n]\}\cup [\pi_1(X), \pi_1(X)]\rangle$, and so this kernel is also analytic.

Now the kernels of the maps $$\pi_1(X) \rightarrow H_1(X)/(\bigoplus_{m=0}^k\mathbb{Z}\oplus H_1(X)_{p_0} \oplus\cdots\oplus H_1(X)_{p_j})$$ are all analytic, and their union letting $j \rightarrow \infty$ is all of $\pi_1(X)$.  Thus by Lemma \ref{biglemma} part (4) we know this ascending sequence eventually stabilizes.  Then we have $H_1(X) \simeq \bigoplus_{m=0}^k\mathbb{Z}\oplus H_1(X)_{p_0} \oplus\cdots\oplus H_1(X)_{p_j}$ for some $j\in \omega$.  For a torsion abelian group $A$ write $A_{p^{q}}$ for the subgroup of those elements whose order divides $p^q$, where $p$ is a prime.  We similarly have that the kernels of the maps $$\pi_1(X) \rightarrow H_1(X)/(\bigoplus_{m=0}^k\mathbb{Z}\oplus H_1(X)_{p_0^q} \oplus H_1(X)_{p_1}\cdots\oplus H_1(X)_{p_j})$$ are analytic for all $q\in \omega\setminus \{0\}$.  The union of all these subgroups is the whole of $\pi_1(X)$, so the sequence must stabilize.  Proceeding similarly for $p_1, p_2,$ etc., we get that $\tor(H_1(X))$ is a finite direct sum of $p$-subgroups, each of bounded power.  Thus $\tor(H_1(X))$ is a torsion group of bounded power, and so is a direct sum of cyclic groups by a theorem of Pr{\"u}fer (see \cite{Fu} Theorem 11.2).
\end{proof}

A theorem of Kulikov (see Theorem 12.2 of \cite{Fu}) states that subgroups of direct sums of cyclic groups are direct sums of cyclic groups. Thus if $\card(\torfree(H_1(X)))$ $<2^{\aleph_0}$ we see that $H_1(X)$ has no nontrivial divisible subgroup (a group is divisible if each element has $n$-th roots for every $n \in \omega \setminus \{0\}$).

Since the divisible abelian groups are injective, each torsion-free abelian group $A$ decomposes as a direct sum $A = \Div(A) \oplus \red(A)$ where $\Div(A)$ is the maximal divisible group in $A$ and $R$ is a reduced subgroup (i.e. contains no nontrivial divisible subgroups.)  The reduced subgroup is not necessarily unique as a subgroup of $A$, but is unique up to isomorphism as it is isomorphic to $A/\Div(A)$.  

\begin{theorem}\label{reduced}  If $X$ is a Peano continuum then $\red(\torfree(H_1(X)))$ is either a free abelian group of finite rank or of cardinality $2^{\aleph_0}$.
\end{theorem}

\begin{proof}  By Lemma \ref{biglemma} part (3) we need only show that the kernel $K$ of the map $\pi_1(X) \rightarrow \red(\torfree(H_1(X)))$ is analytic.  We have already seen that the kernel $K_1$ of the map $\pi_1(X) \rightarrow \torfree(H_1(X))$ is analytic.  Now $l$ is in $\bigcup K$ if and only $(\forall n \in \omega\setminus \{0\})(\exists l_1 \in L_x)[(l_1)^nl^{-1} \in \bigcup K_1]$.  Then $K$ is an analytic subgroup as a countable intersection of countinuous preimages of analytic subsets of $L_x$.
\end{proof}

From Example \ref{primeexample} where $\pi_1(X) \simeq \prod_{p\in P}\mathbb{Z}/p$ we know each element of the group $\torfree(H_1(X))$ is divisible, by the same explanation as given for $(1,1, \ldots)$.  Thus $\torfree(H_1(X))$ is a torsion-free divisible group of cardinality $2^{\aleph_0}$.  Thus by considering $\torfree(H_1(X))$ as a vector space over $\mathbb{Q}$ we may select a basis, so that $\torfree(H_1(X)) \simeq \bigoplus_{2^{\aleph_0}} \mathbb{Q}$.  In this case the reduced summand of $\torfree(H_1(X))$ is trivial, and by taking a product of $X$ with a finite dimensional torus $T^n$ we easily get that $\torfree(H_1(X \times T^n)) \simeq \bigoplus_{2^{\aleph_0}} \mathbb{Q} \oplus \prod_{i=0}^{n-1} \mathbb{Z}$.  By taking the product of $X$ with the infinite torus $T^{\infty}$ we get $\torfree(H_1(X \times T^n)) \simeq \bigoplus_{2^{\aleph_0}} \mathbb{Q} \oplus \prod_{n\in \omega} \mathbb{Z}$, so all cases of Theorem \ref{reduced} obtain.

Given an abelian group $A$ there is a largest quotient $A/B$ with no nontrivial infinitely divisible elements, i.e. elements $a \neq 0$ such that for some $b\in A/B$we have $nb =a$ in $A/B$ for infinitely many $n\in \omega$.  To see this, let $S_0$ be the set of infinitely divisible elements in $A$.  Let $S_1$ be the set of those elements in $A$ which map under the quotient map $A \rightarrow A/\langle S_0\rangle$ to an infinitely divisible element.  In general let $S_{\alpha+1}$ be the set of those elements which map to an infinitely divisible element under the quotient map $A\mapsto A/\langle S_{\alpha} \rangle$ and $S_{\beta} = \bigcup_{\alpha<\beta} S_{\alpha}$ for $\beta$ a limit ordinal.  It is clear that $S_0 \subseteq \langle S_0\rangle \subseteq S_1 \subseteq \langle S_1\rangle \subseteq \cdots$.  The process must stabilize eventually since $A$ is a set, say at step $\beta$.  The subgroup $B =  \bigcup_{\alpha<\beta} \langle S_{\alpha}\rangle$ is evidently such that $A/B$ has no infinitely divisible elements and any homomorphism to an abelian group with no infinitely divisible elements must evidently contain $B$.   Let $\inffree(A)$ denote this maximal quotient with no infinitely divisible elements.

For the next result we state Theorem 4.4 in \cite{CC}:

\begin{theorem*}  Let $X$ be a topological space, let $\phi:\pi_1(X, x_0) \rightarrow L$ a be a homomorphism to the group $L$, $U_0\supseteq U_1 \supseteq \cdots$ be a countable local basis for $X$ at $x_0$ and $G_i$ be the image of the natural map of $\pi_1(U_i, x_0)$ into $\pi_1(X, x_0)$.  If $L$ is of cardinality $<2^{\aleph_0}$ and abelian with no infinitely divisible elements then $\phi(G_n) = 0$ for some $n\in \omega$.
\end{theorem*}

\begin{theorem}\label{inffree}  If $X$ is a Peano continuum with $\card(\inffree(H_1(X)))<2^{\aleph_0}$ then $\inffree(H_1(X))$ is a free abelian group of finite rank.

\end{theorem}

\begin{proof}  The homomorphism $\pi_1(X) \rightarrow \inffree(H_1(X))$ is, by the previously stated theorem, such that given any point $x\in X$ there is a neighborhood $O_x$ such that any loop in $O_x$ maps trivially under the map.  This implies that the kernel is open (see Lemma 2.8 in \cite{Co}), and so $\inffree(H_1(X))$ is a finitely generated abelian group with no torsion by Lemma \ref{biglemma} part (3).
\end{proof}

We conclude this section with a dichotomy theorem on the number of homomorphisms to minimal countable fields.

\begin{theorem}  If $X$ is a Peano continuum then $\Hom(\pi_1(X), \mathbb{Z}/p)$ is either finite or of cardinality $2^{2^{\aleph_0}}$ and $\Hom(\pi_1(X), \mathbb{Q})$ is either countable or of cardinality $2^{2^{\aleph_0}}$.
\end{theorem}

\begin{proof}  Notice that the set $\{l^p: l\in L_x\}$ is an analytic set in $L_x$ as the continuous image of $L_x$.  The subgroup $\langle [\{l^p: l\in L_x\}]\cup [\pi_1(X), \pi_1(X)]\rangle$ is analytic, being generated by two analytic sets.  Any homomorphism $\pi_1(X) \rightarrow \mathbb{Z}/p$ must factor through a homomorphism $\pi_1(X)/\langle [\{l^p: l\in L_x\}]\cup [\pi_1(X), \pi_1(X)]\rangle \rightarrow \mathbb{Z}/p$.  Thus the homomorphisms $\pi_1(X) \rightarrow \mathbb{Z}/p$  are in correspondence with the homomorphisms $\pi_1(X)/\langle [\{l^p: l\in L_x\}]\cup [\pi_1(X), \pi_1(X)]\rangle \rightarrow \mathbb{Z}/p$.  The group $\pi_1(X)/\langle [\{l^p: l\in L_x\}]\cup [\pi_1(X), \pi_1(X)]\rangle$ is a vector space over $\mathbb{Z}/p$, so we may pick a basis and write $\pi_1(X)/\langle [\{l^p: l\in L_x\}]\cup [\pi_1(X), \pi_1(X)]\rangle \simeq \bigoplus_{T} \mathbb{Z}/p$.  By Lemma \ref{biglemma} part (3) we know $T$ is either finite or of cardinality $2^{\aleph_0}$.  In case $T$ is finite, there are finitely many homomorphisms from $\bigoplus_{T} \mathbb{Z}/p$ to $\mathbb{Z}/p$, and in case $T$ is of cardinality $2^{\aleph_0}$ there are $2^{2^{\aleph_0}}$ many.

For the claim regarding $\mathbb{Q}$ we notice that any homomorphism $\pi_1(X) \rightarrow \mathbb{Q}$ must factor through $\torfree(H_1(X))$.  We have seen that $\torfree(H_1(X))$ is either a finite rank free group or of cardinality $2^{\aleph_0}$ (Theorem \ref{modtor}).  In case $\torfree(H_1(X))$ is a finite rank free group we have countable $\Hom(\pi_1(X), \mathbb{Q})$.  Suppose $\torfree(H_1(X))$ has cardinality $2^{\aleph_0}$.  Recall that the rank of an abelian group $A$ is the rank of the largest free abelian group in $A$, and satisfies the inequality $\card(A) \leq \rank(A) \aleph_0$.  Thus $\torfree(H_1(X))$ has a free abelian subgroup $F$ of rank $2^{\aleph_0}$.  There are $2^{2^{\aleph_0}}$ many homomorphisms from $F$ to $\mathbb{Q}$, and since $\mathbb{Q}$ is an injective $\mathbb{Z}$-module, each of these homomorphisms extends to all of $\torfree(H_1(X))$.  This establishes $2^{2^{\aleph_0}}$ many distinct homomorphisms from $\torfree(H_1(X))$ to $\mathbb{Q}$, and there cannot be more than $2^{2^{\aleph_0}}$, so we are done.
\end{proof}

\end{section}

\begin{section}{A strong abelianization}\label{Astrongabelianization}

We introduce a notion of strong abelianization:

\begin{definition}  If $X$ is a path connected metrizable space, define $\h(X)$ to be the quotient $\pi_1(X)/\overline{[\pi_1(X), \pi_1(X)]}$ where $\overline{G}$ is defined to be the smallest closed subgroup containing $G \leq \pi_1(X)$.
\end{definition}

The strong abelianization $\h(X)$ corresponds more with our intuition of abelianization for the fundamental group of a space and we shall see that it is easier to understand.  We establish some functorial properties.

\begin{theorem}\label{products}  Suppose that $X_n$ is a path connected metrizable space for each $n\in\omega$.  Then $\h(\prod_n X_n) \simeq \prod_{n}\h(X_n)$.
\end{theorem}

\begin{proof}  For each $n\in \omega$ fix a point $x_n\in X_n$ and assume without loss of generality that $\diam(X_n)\leq 2^{-n}$. Let $x = \{x_n\}_{n\in \omega}\in \prod_{\omega}X_n = X$.  The loop space $L_x$ is homeomorphic to the space $\prod_{n\in \omega} L_{x_n}$ (see Lemma 3.8 in \cite{Co}) and so can be metrized with the metric inherited by the product metric.  Let $p_n: L_x \rightarrow L_{x_n}$ denote projection to the $n$-th coordinate.

Supposing $l, l'\in L_x$ we have that the loop $l*l'*l^{-1}*(l')^{-1}$ projects under $p_n$ to $p_n(l)*p_n(l')*p_n(l)^{-1}*p_n(l')^{-1}$.  Thus by taking products we see that the commutator subgroup $[\pi_1(X), \pi_1(X)]$ is naturally a subgroup of $\prod_n [\pi_1(X_n), \pi_1(X_n)]$.  For each $n\in \omega$ the map $\iota_n:  L_{x_n} \rightarrow L_x$ which takes a loop $l\in L_{x_n}$ to the loop $l'\in L_x$ such that $p_n(l'(s)) = l(s)$ and $p_m(l'(s)) = x_m$ for $m\neq n$ demonstrates the inclusion $\bigoplus_n [\pi_1(X_n), \pi_1(X_n)] \leq [\pi_1(X), \pi_1(X)]$.

For each $n$ we have that the continuous preimage $p_{n*}^{-1}(\overline{[\pi_1(X_n), \pi_1(X_n)]}) = \prod_{m<n} \pi_1(X_m) \times \overline{[\pi_1(X_n), \pi_1(X_n)]} \times \prod_{m>n}\pi_1(X_m)$ is closed and contains the subgroup $\prod_n [\pi_1(X_n), \pi_1(X_n)]$.  Hence the intersection over $n\in \omega$ of all such groups,  $\prod_n \overline{[\pi_1(X_n), \pi_1(X_n)]}$, contains the subgroup $\overline{[\pi_1(X), \pi_1(X)]}$.  On the other hand we have the inclusion $\overline{\bigoplus_n [\pi_1(X_n), \pi_1(X_n)]} \leq \overline{[\pi_1(X), \pi_1(X)]}$.  Since the set of loops $\bigcup \bigoplus_n [\pi_1(X_n), \pi_1(X_n)]$ is dense in the set of loops $\bigcup \prod_n [\pi_1(X_n), \pi_1(X_n)]$ we see that $\overline{\bigoplus_n [\pi_1(X_n), \pi_1(X_n)]} = \overline{\prod_n [\pi_1(X_n),\pi_1(X_n)]}$.  The equality $$\overline{\prod_n [\pi_1(X_n),\pi_1(X_n)]} = \prod_n\overline{[\pi_1(X_n),\pi_1(X_n)]}$$ is similarly clear and so we see that $\prod_n \overline{[\pi_1(X_n), \pi_1(X_n)]} = \overline{[\pi_1(X), \pi_1(X)]}$ from which we get the isomorphism $\h(\prod_n X_n) \simeq \prod_{n}\h(X)$.
\end{proof}

\begin{example}  The circle $S^1$ is a semilocally simply connected metric space whose fundamental group is isomorphic to $\mathbb{Z}$, and so we can compute the strong abelianization of the infinite torus by the above theorem: $\h(T^{\infty}) \simeq \prod_{n\in \omega} \h(S^1) \simeq \prod_{n\in\omega} \mathbb{Z}$.
\end{example}

The standard abelianization of abstract groups does not behave in nearly so nice a manner.  If $G = \prod_{n\in \omega} G_n$ is a product of groups, then the abelianization of $G$ needn't be the product of abelianizations of the $G_n$, since an element of $[G, G]$ needs to have finite commutator length and elements of $\prod_{n}[G_n, G_n]$ can have infinite commutator length.

\begin{definition}  Let $X_n$ be a sequence of metrizable spaces with distinguished points $x_n$.  Using a cutoff metric if necessary we endow each space with metric $d_n$ such that $\diam(X_n) \leq 2^{-n}$.  Define the shrinking wedge of spaces $\bigvee_{n\in \omega}^s (X_n, x_n)$ to be the set which identifies the points $x_n$, with topology given by the metric $d(y,z) = \begin{cases}d_n(y,z)$ if $y,z\in X_n\\ d_n(y, x_n) + d_m(z, x_m)$ if $y\in X_n\setminus \{x_n\}, z\in X_m\setminus\{x_m\}$ with $m\neq n\end{cases}$.  It is not difficult to see that the topology does not depend on the metrics chosen and is homeomorphic under any reordering of the index set.  If we let all but finitely many of the spaces $X_n$ be a single point then we obtain the standard (finitary) wedge of spaces.
\end{definition}

\begin{theorem}\label{shrinkingwedge} If $\{(X_n, x_n)\}_{n\in \omega}$ is a collection of path connected, metrizable pointed spaces then $\h(\bigvee_{n\in \omega}^s (X_n, x_n)) \simeq \prod_{n\in\omega}\h(X_n)$.
\end{theorem}

\begin{proof}  Let $(X, x) = \bigvee_{n\in \omega}^s (X_n, x_n)$ and $f:X \rightarrow \prod_{n\in \omega} X_n$ be the obvious map.  That $f_*$ is onto follows from the fact that $d_n \leq 2^{-n}$.  As $f$ is continuous we know the subgroup $(f_*)^{-1}(\overline{[\pi_1(\prod_{n\in\omega}X_n), \pi_1(\prod_{n\in\omega}X_n)]})$ is a closed subgroup of $\pi_1(X, x)$ and must contain the commutator subgroup.  For each $m$ let $r_m:(X, x) \rightarrow (X_m, x_m)$ be the retraction which takes all subspaces $X_n$ to the point $x_m$ whenever $n \neq m$.

Suppose $l\in L_x$.  For each $m\in \omega$ let $\mathcal{I}_m$ be the set of those maximal closed intervals $[a,b] \subseteq [0, 1]$ with nonempty interior such that $l|[a, b]$ is a loop in $X_m$ with $a = \inf\{s\in [a,b]: l(s) \neq x_m\}$ and $b = \sup\{s\in [a,b]: l(s) \neq x_m\}$.  Write $\mathcal{I}_m = \{I_{m, 0}, I_{m, 1}, I_{m, 2}, \ldots \}$ so that the length of $I_{m, k}$ is at least as great as the length of $I_{m, k+1}$.  Then each $\mathcal{I}_m$ consists of disjoint closed intervals and the collection $\mathcal{I} = \bigcup_{m} \mathcal{I}_m$ consists of nonoverlapping intervals, which has a natural ordering by comparing elements in the interior of $I$ and $I'$ under the natural ordering of $[0,1]$.  The loop $l$ is equivalent over $[\pi_1(X, x), \pi_1(X, x)]$ to the loop $l'$ which has $\inf(I_{0, 0}) = 0$, and $\sup(I_{0, 0})$ equal to the length of $I_{0,0}$ and all the other relative positions of the elements of $\mathcal{I}$ unchanged.  This loop $l'$ is in turn equivalent over $[\pi_1(X, x), \pi_1(X, x)]$ to the loop $l''$ which has $I_{0,0}$ in the same position as $l'$ has and $I_{0, 1}$ immediately to the right of $I_{0,0}$ (now $I_{0, 0}$ and $I_{0, 1}$ are no longer disjoint).  Continuing in this manner we get loops $l'''$, etc.,  which are equivalent to $l$ over $[\pi_1(X, x), \pi_1(X, x)]$ with more and more of the elements of $\mathcal{I}_0$ adjacent to each other.  Taking the limit of these loops we see that $l$ is equivalent to a loop $l_0$ over $\overline{[\pi_1(X, x), \pi_1(X, x)]}$ such that $l_0|[0, s_0]$ is a loop lying in $X_0$ with $l_0(s) \notin X_0\setminus\{x_0\}$  for all $s\in [s_0, 1]$.  Now we perform the same process to the loop $l_0|[s_0, 1]$ to obtain a loop $l_1$ which is equivalent over $\overline{[\pi_1(X, x), \pi_1(X, x)]}$ to $l$ such that $l_0|[0,s_0] = l_1|[0, s_0]$ and $l_1|[s_0, s_1]$ is a loop in $X_1$ and $l_1(s) \notin X_1\setminus\{x_1\}$ for all $s>s_1$.  Continue in this process to get loops $l_2, l_3,\cdots$ such that the analogous relations hold.  Since the lengths of the elements of $\mathcal{I}$ must add to a number at most $1$, we may take a limit again and see that $l$ is equivalent over $\overline{[\pi_1(X, x), \pi_1(X, x)]}$ to a loop $\tilde{l}$ such that $\tilde{l}|[0, s_0]$ is a loop in $X_0$ and $\tilde{l}(s) \notin X_0\setminus \{x_0\}$ for all $s>s_0$, $\tilde{l}|[s_0, s_1]$ is a loop in $X_1$ such that $\tilde{l}(s) \notin X_1\setminus\{x_1\}$ for $s\notin[s_0, s_1]$, and in general $\tilde{l}|[s_{n-1}, s_{n}]$ is a loop in $X_n$ such that $\tilde{l}(s) \notin X_n\setminus\{x_n\}$ for $s\notin [s_{n-1}, s_n]$.

Now consider a loop $l\in \bigcup f_*^{-1}(\overline{[\pi_1(\prod_{n\in\omega}X_n), \pi_1(\prod_{n\in\omega}X_n)]})$ where without loss of generality $l$ is of the form $\tilde{l}$ as in the preceeding paragraph.  We have $p_0\circ f\circ l|[0, s_0] = r_0\circ l|[0,s_0]$ is a loop in $\overline{[\pi_1(X_0), \pi_1(X_0)]}$ by the proof of Theorem \ref{products}.  Then $l|[0,s_0]$ is a loop in $\overline{[\pi_1(X_0), \pi_1(X_0)]}$ by considering the inclusion map $(X_0, x_0)\rightarrow (X, x)$.  Then $l$ is equivalent over $[\pi_1(X, x), \pi_1(X, x)]$ to a loop $l_0$ such that $l|[s_0, 1] = l_0|[s_0, 1]$ and $l(s) = x$ for all $s\in [0,s_0]$.  Continuing in this fashion and taking a limit we see that $l$ is equivalent over $\overline{[\pi_1(X, x), \pi_1(X, x)]}$ to the constant loop at $x$.  This gives us the reverse containment.
\end{proof}

\begin{example}\label{E}  Recall the Hawaiian Earring $E=\bigcup_{n\in \omega} C((0, \frac{1}{n+2}),\frac{1}{n+2})\subseteq \mathbb{R}^2$, with $C(p,r)$ the circle centered at point $p$ of radius $r$.  Clearly $E$ is a shrinking wedge of countably many circles.  Thus $\h(E) \simeq \prod_{n\in \omega} \mathbb{Z}$.  This computation concides with the standard short proof that $H_1(E)$ is uncountable which is given by mapping onto the fundamental group of the infinite torus.  The first homology was computed by K. Eda in \cite{E1} to be isomorphic to $\prod_{n\in \omega}\mathbb{Z} \oplus  \bigoplus_{2^{\aleph_0}} \mathbb{Q} \oplus \bigoplus_p A_p$ where $A_p$ is the $p$-dic completion of the free abelian group of rank $2^{\aleph_0}$.
\end{example}

We note that first homology does not behave this nicely even under a wedge of two spaces, by observing the following theorem of Eda (see \cite{E2}):

\begin{theorem*}  Letting $(X, x)$ and $(Y, y)$ be arbitrary pointed spaces we have that $H_1((X, x)\vee(Y, y)) \simeq H_1(X) \oplus H_1(Y) \oplus H_1(C(X, x)\vee C(Y, y))$ where $C(W, w)$ denotes the cone of the space $W$ with distinguished point $(w, 0)$ in the cone.
\end{theorem*}

The homology group $H_1(C(X, x)\vee C(Y, y))$ is often nontrivial, so the standard first homology is in some ways less wieldy.
\end{section}

\begin{section}{Borel subgroups}\label{Borelsubgroups}

We show that the fundamental group of the Hawaiiam earring $E$ (see Example \ref{E}) exhibits subgroups of an arbitrarily high Borel complexity, as well as a subgroup which is analytic and not Borel.  For a Polish group $G$ we say $H \leq G$ is of pointclass $\Po$ if and only if $H$ is of pointclass $\Po$ as a subset of $G$.

Select a point $x$ in the circle $S^1$ and let $\overline{x} = (x, x, \ldots)\in \prod_{\omega}S^1 = T^{\infty}$.  Let $l\in L_{x, S^1}$ be a loop such that $[l]$ generates $\pi_1(S^1, x) \simeq \mathbb{Z}$.  Let $g:L_{x, S^1} \rightarrow \mathbb{Z}$ be given by $[l] \mapsto 1$ and mapping all loops in all other equivalence classes by $[l^m] \mapsto m$ for all $m\in \mathbb{Z}$.  Viewing $\mathbb{Z}$ as a discrete topological space (which endows $\mathbb{Z}$ with a Polish group topology), the map $g$ is continuous since each homotopy class in $\pi_1(S^1, x)$ is clopen in the space $L_{x, S^1}$ (this follows from the fact that $S^1$ is locally path connected and semi-locally simply connected, see Proposition 2.10 in \cite{Co}).  Letting $j: \mathbb{Z} \rightarrow L_{x, S^1}$ be defined by $j(m) = l^{m}$, it is clear that $j$ is also continuous and that $g\circ j$ is the identity map on $\mathbb{Z}$ and $j\circ g$ selects the representative $l^m \in [l^m]$.  Defining $\overline{g}: \prod_{\omega} L_{x, S^1} \rightarrow \prod_{\omega}\mathbb{Z}$ and $\overline{j}:\prod_{\omega}\mathbb{Z} \rightarrow  \prod_{\omega} L_{x, S^1}$ coordinatewise it is clear that $\overline{g}$ and $\overline{j}$ are continuous and that $\overline{g}\circ \overline{j}$ is identity on $\prod_{\omega}\mathbb{Z}$ and $\overline{j}\circ\overline{g}$ selects the representative $(l^{m_0}, l^{m_1}, \ldots) \in [(l^{m_0}, l^{m_1}, \ldots)]$.  It is a straightforward exercise to check that the loop space $L_{\overline{x}, T^{\infty}}$ is homeomorphic to the product $\prod_{\omega} L_{x, S^1}$ in the obvious way (see Lemma 3.8 in \cite{Co}).

\begin{lemma}\label{Borellemma1}  Suppose $\Po$ is a pointclass defined for Polish spaces which is closed under continuous preimages.  Then $G \leq \prod_{\omega}\mathbb{Z}$ is $\Po$ if and only if $G$ is $\Po$ as a subgroup of $\pi_1(T^{\infty}, \overline{x})$.
\end{lemma}

\begin{proof}  Suppose $G$ is a $\Po$ subgroup of the Polish group $\prod_{\omega}\mathbb{Z}$.  The set of loops in $\pi_1(T^{\infty}, \overline{x})$ corresponding to $G$ is precisely $\overline{g}^{-1}(G)$, so $G$ is $\Po$ when considered a subgroup of $\pi_1(T^{\infty}, \overline{x})$.

Supposing $G$ to be $\Po$ as a subgroup of $\pi_1(T^{\infty}, \overline{x})$, the set $\overline{j}^{-1}(\bigcup G)\subseteq \prod_{\omega}\mathbb{Z}$ corresponds precisely to $G$ considered as a subset of $\prod_{\omega}\mathbb{Z}$, so $G$ is $\Po$ in the Polish group $\prod_{\omega}\mathbb{Z}$.
\end{proof}

Next we draw a correspondence between subgroups of $\pi_1(T^{\infty}, \overline{x})$ and certain of the subgroups of $\pi_1(E, (0,0))$.  Let $f:E \rightarrow T^{\infty} = \prod_{\omega} S^1$ be the standard mapping from the Hawaiian earring defined by mapping the $n$-th circle to the circle in $T^{\infty}$ at the $n$-th coordinate.  This is a continuous, onto mapping.  The map $f$ defines a continuous $\overline{f}:L_{(0,0),E}\rightarrow L_{\overline{x},T^{\infty}}$ by composition: $\overline{f}(l) = f\circ l$.  Define $\overline{k}:L_{\overline{x}, T^{\infty}} \rightarrow L_{(0,0), E}$ by having $\overline{k}(l)|[1-2^{-n}, 1-2^{-n - 1}]$ trace out the loop $p_n\circ l$ on the $n$-th circle of the Hawaiian earring and having $\overline{k}(l)(1) = (0,0)$.  That $\overline{k}$ is continuous is easy to check.  We also have the relations $\overline{f} \circ \overline{k}(l) \in [l]$.

\begin{lemma}\label{Borellemma2}  Suppose $\Po$ is a pointclass defined for Polish spaces which is closed under continuous preimages.  Then $G\leq \pi_1(T^{\infty}, \overline{x})$ is $\Po$ if and only if $f_*^{-1}(G)$ is $\Po$. 
\end{lemma}

\begin{proof}  If $G \leq \pi_1(T^{\infty}, \overline{x})$ is $\Po$ then $\bigcup f_*^{-1}(G) = \overline{f}^{-1}(\bigcup G)$ is also $\Po$.  Now suppose that $f_*^{-1}(G)$ is $\Po$.  We shall be done if we prove the equality $\overline{k}^{-1}(\bigcup f_*^{-1}(G)) = \bigcup G$.  If $l\in \bigcup G$ we already noted above that $\overline{f}\circ \overline{k}(l) \in [l] \in G$ so that $\overline{k}(l) \in \bigcup f_*^{-1}(G)$ and $l \in \overline{k}^{-1}(\bigcup f_*^{-1}(G))$.  On the other hand, if $l\in \overline{k}^{-1}(\bigcup f_*^{-1}(G))$ we know $\overline{k}(l) \in \bigcup f_*^{-1}(G)$ so that $[\overline{k}(l)]\in  f_*^{-1}(G)$ and $[f\circ \overline{k}(l)]\in G$.  But $[f\circ\overline{k}(l)] = [l]$, so $l\in \bigcup G$ and we are done.

\end{proof}

We recall the Borel pointclass hierarchy (see \cite{Ke}).  If $Z$ is a Polish space let $\Sigma_1^0(Z)$ denote the collection of open subsets of $X$, $\Pi_1^0(Z)$ the collection of closed subsets and $\Delta_1^0(Z) = \Sigma_1^0(Z) \cap \Pi_1^0$.  For each ordinal $1< \gamma<\omega_1$ define the following:

\begin{enumerate}  \item $\Sigma_{\gamma}^0(Z)$ is the collection of countable unions of sets in $\bigcup_{\epsilon<\gamma} \Pi_{\epsilon}^0(Z)$

\item $\Pi_{\gamma}^0(Z)$ is the collection of countable intersections of sets in $\bigcup_{\epsilon<\gamma} \Sigma_{\epsilon}^0(Z)$

\item $\Delta_{\gamma}^0(Z) = \Pi_{\gamma}^0(Z) \cap \Sigma_{\gamma}^0(Z)$
\end{enumerate}

These pointclasses arrange neatly into an array scheme:

\vspace{.3in}

\begin{center}
$\begin{matrix} & & \Sigma_1^0(Z) & & & & \Sigma_2^0(Z) & &  \cdots\\ \\ \\ \Delta_1^0(Z) & &  & &\Delta_2^0(Z) & &  \cdots\\ \\ \\ & & \Pi_1^0(Z) & & & & \Pi_2^0(Z) & &  \cdots

\end{matrix}
$
\end{center}

\vspace{.3in}

Each pointclass contains all pointclasses to the left, and if $Z$ is an uncountable Polish space the containments are strict.  The class of Borel subsets of $Z$ is easily checked to be  $\bigcup_{\gamma<\omega_1} \Sigma_{\gamma}^0(Z)$, which is equal to $\bigcup_{\gamma<\omega_1} \Pi_{\gamma}^0(Z)$.  These pointclasses give a natural way of organizing the Borel sets according to complexity.  We shall say a subset $W\subseteq Z$ is a \textbf{true} $\Po$ set provided it is a $\Po$ set and is not in any of the pointclasses to the left of $\Po$.  Each of these pointclasses is closed under continuous preimages.

The class $\Sigma_1^1(Z)$ is defined to be the class of all analytic subsets of $Z$, $\Pi_1^1(Z)$ is the class of complements of analytic sets and $\Delta_1^1(Z) = \Sigma_1^1(Z) \cap \Pi_1^1(Z)$.  As all Borel sets are analytic we know the class of Borel subsets of $Z$ is contained in $\Delta_1^1(Z)$, and a theorem of Suslin states that in fact $\Delta_1^1(Z)$ is precisely the class of Borel sets.  Analogously a subset $W \subseteq Z$ is \textbf{true analytic} provided $W$ is analytic and not Borel.

We have the following theorem of Farah and Solecki (Theorem 2.1 in \cite{FS}):

\begin{theorem*}  If $\Gamma$ is an uncountable Polish group then for each $\omega_1>\gamma \geq 2$ there exists a subgroup which is true $\Sigma_{\gamma}^0$ in $\Gamma$ and for each $\omega_1 > \gamma \neq 2$ there is a subgroup which is true $\Pi_{\gamma}^0$ as a subspace of $\Gamma$.
\end{theorem*}

As a consequence of Theorem 1.2 in \cite{DL} we also get the following:

\begin{theorem*}  If $\Gamma$ is an uncountable abelian Polish group then $\Gamma$ has a subgroup which is true analytic.
\end{theorem*}

These results lend a robustness to the theory of Borel and analytic subgroups of Polish groups.  As a consequence, a comparable robustness holds for subgroups of the Hawaiian earring fundamental group, which is summarized in the following:

\begin{B}  The fundamental group $\pi_1(E)$ has normal subgroups of the following types:

\begin{enumerate}\item true $\Sigma_{\gamma}^0$ for $\omega_1>\gamma \geq 2$

\item true $\Pi_{\gamma}^0$  for each $\omega_1 > \gamma \neq 2$

\item true analytic
\end{enumerate}
\end{B}

\begin{proof}  The items (1) and (2) follow directly from the stated theorem of Farah and Solecki, for if $G \leq \prod_{n\in \omega}\mathbb{Z}$ is true $\Po$, then considering $G$ as a subgroup of $\pi_1(T^{\infty})$ we have that $G$ is true $\Po$, and so $f_*^{-1}(G)$ must also be true $\Po$ (by Lemmas \ref{Borellemma1} and \ref{Borellemma2}).  Item (3) follows in the same way, from the consequence of \cite{DL}.
\end{proof}

\end{section}

\end{document}